\documentclass[12pt]{amsart}

\usepackage{amssymb, xcolor}%
\usepackage
{amsaddr}
\usepackage{geometry}%
\usepackage{graphicx}

\begingroup%
\theoremstyle{plain}
\newtheorem{prop}{Proposition}[section]
\newtheorem{corollary}[prop]{Corollary}%
\newtheorem{lemma}[prop]{Lemma}%
\newtheorem{remark}[prop]{{Remark}}%
\newtheorem{definition}[prop]{{Definition}}%
\newtheorem{defn}[prop]{{Definition}}%
\newtheorem{question}[prop]{Question}%
\endgroup%

\theoremstyle{plain}
\newtheorem{theorem}[prop]{Theorem}%

\theoremstyle{break}
\newtheorem{open?}[prop]{{Open Question}}

\theoremstyle{break}
\def\NN{\mathbb N}

\def\ZZ{\mathbb Z}

\def\endproof{\hfill$\square$}

\def\scalar{$s$-multi\-plicative}

\def\GIv{G_{(v,I)}}
\def\GIw{G_{(w,I)}}
\def\GJv{G_{(v,J)}}
\def\Gom{G_{\omega+}}
\def\omG{G_{\omega-}}
\def\Gomp{G_{\sigma+}}
\def\ompG{G_{\sigma-}}

\keywords{scale function, tidy subgroup, homogeneous tree, tree ends}
\begin{document}

\title[Scale-multiplicative semigroups and geometry]%
{Scale-multiplicative semigroups and geometry:  automorphism groups of trees}
\author{U. Baumgartner} 
\address{School of Mathematics and Applied Statistics, 
          University of Wollongong,
           Wollongong NSW 2522,
           Australia,\\ 
          Tel.: +61-2 4221 3388 (school office)
           Fax: +61-2 4221 4845
}
\email{Udo\_Baumgartner@uow.edu.au}           
\author{J. Ramagge}
\address{School of Mathematics and Applied Statistics, 
          University of Wollongong,
           Wollongong NSW 2522,
           Australia,\\ 
          Tel.: +61-2 4221 3388 (school office)
           Fax: +61-2 4221 4845
}
\email{ramagge@uow.edu.au}
\author{G. A. Willis}
\address{School of Mathematical and Physical Sciences,
              The University of Newcastle, University Drive, Building V,
              Callaghan, NSW 2308, Australia, \\
              Tel.: +61-2 4921 5546, 
              Fax: +61-2 4921 6898, 
}
\email{George.Willis@newcastle.edu.au}           
\thanks{Supported by ARC Discovery Project DP0984342}

\date{\today}
\maketitle

\begin{abstract}
A scale-multiplicative semigroup in a totally disconnected, locally compact group $G$ is one for which the restriction of the scale function on $G$ is multiplicative. The maximal scale-multiplicative semigroups in 
groups 
acting 2-transitively on the set of ends 
of trees 
without leaves 
are determined in this paper and shown to correspond to geometric features of the tree.
\end{abstract}

\section{Introduction}
\label{sec:intro}

The scale function on a totally disconnected, locally compact group,~$G$, was introduced in~\cite{Willis94} as a tool for use in the proof of a conjecture made in~\cite{HofMukh}. The scale, $s(x)$, of an element $x$ in $G$ is a positive integer that by~\cite[Theorem~3.1]{Willis01} is equal to
\begin{equation}
\label{eq:scale}
s(x) := \min\left\{ [xVx^{-1} : xVx^{-1}\cap V] \; \colon V\leqslant G,\ V\text{ compact and open}\right\}.
\end{equation}
The scale is attained in~\eqref{eq:scale} because it is the minimum of a set of positive integers, which may be seen as follows. Compact open subgroups exist because $G$ has a base of neighborhoods of the identity comprising such subgroups (see~\cite{Dant}, \cite[Theorem~II.2.3]{MontZip} or~\cite[Theorem~II.7.7]{HewRoss}). Then, for $V$ compact and open, $xVx^{-1}\cap V$ is an open subgroup of $xVx^{-1}$, which is compact, whence $[xVx^{-1} : xVx^{-1}\cap V]$ is a positive integer. 

Since the scale function takes positive integer values, it cannot be multiplicative on~$G$, or on subgroups of $G$, unless $s(x) = 1$ for every~$x$. This rarely occurs because  $s(x) = 1 = s(x^{-1})$ only if~$x$ normalizes some compact, open subgroup of~$G$. 
There is no obstruction to the scale being multiplicative on a semigroup, however, and it indeed is multiplicative on singly generated semigroups since $s(x^n) = s(x)^n$ for every $n\in\ZZ^+$ and $x$ in $G$ by~\cite[Corollary~3]{Willis94}. A semigroup on which the scale function is multiplicative will be called \emph{scale-multiplicative} or \emph{\scalar}. 

This paper investigates the maximal \scalar\ semigroups in totally disconnected, locally compact groups. The reason for doing so is that every \scalar\ semigroup is contained in a maximal one, as a Zorn's Lemma argument shows. It is our contention that the set of all maximal \scalar\ semigroups in $G$, and the relations between them, form a structure space that encapsulates important information about $G$. 

As a first step towards realizing the goal of producing a relational structure from maximal \scalar\ semigroups, we determine all such semigroups in the automorphism group of certain trees and relate them to the geometry of the tree. In doing so, this paper extends work of J.~Tits.  
In~\cite{arbre} Tits shows how to reconstruct a tree from the maximal compact, open subgroups in its automorphism group. 
It will be seen that each of these subgroups is also a maximal \scalar\ semigroup but others on which the scale is not always~$1$ are also identified. 
While the maximal compact, open subgroups correspond to vertices of the tree in  Tits' construction, it is seen that these other \scalar\ semigroups correspond to the  ends of the tree and to open sets in the space of ends. In addition to this correspondence with other features, what is gained by shifting attention from compact, open subgroups to \scalar\ semigroups is that a general totally disconnected, locally compact group always has maximal \scalar\ semigroups, but 
need not possess any maximal compact, open subgroups. 

The next steps towards creating a relational structure from maximal \scalar\ semigroups will be to investigate these semigroups in other totally disconnected, locally compact groups. Future work will do this for automorphism groups of buildings, semisimple Lie groups over local fields (where an algebraic rather than geometric description of the semigroups is sought) and a group, studied in~\cite{BaRaWi}, that counters intuitions formed from a focus on Lie groups and buildings. An understanding of these \scalar\ semigroups will assist in formulating the hoped-for general definition of a relational structure. 


The following notational convention will be used: the set of positive integers will be denoted by $\ZZ^+$, and the set of natural numbers (including $0$) by $\NN$. 
The identity of a group, $G$, will be denoted~$e_G$. 

\section{Multiplicative semigroups}

This section defines what it means for a semigroup to be multiplicative over a compact, open subgroup $V$ and to be be scale-multiplicative. 
Semigroups satisfying either of these two conditions will  be called \emph{multiplicative}. 
We show that 
semigroups multiplicative over $V$ 
are 
scale-multiplicative and open, 
from which it will follow that 
there exist 
open, maximal scale-multiplicative semigroups. 
If the scale function 
is not identically~$1$, 
there also exist  
such semigroups 
which are 
non-compact as well. 
That the group inverse produces an involution on the set of multiplicative semigroups is also shown. 

The notion of a semigroup being multiplicative over a compact, open subgroup will be defined first. 

\begin{defn}
Suppose that~$V$ is a compact, open subgroup of a totally disconnected, locally compact group~$G$. A semigroup $\mathcal{S} \subseteq G$ is \emph{multiplicative over $V$} if 
$V\subseteq \mathcal{S}$ and the map 
$s_V:\mathcal{S}\to \ZZ^+$ given by $s_V(x):= [xVx^{-1}\colon xVx^{-1}\cap V]$ 
satisfies $s_V(xy)=s_V(x)s_V(y)$ for all $x,y\in \mathcal{S}$.
\end{defn}
Since the set of semigroups multiplicative over $V$ is closed under increasing unions, each semigroup multiplicative over $V$ is contained in a maximal such semigroup. This maximal semigroup is not unique in general.

Our first result relates multiplicativity over $V$ to the notion of scale-multiplica\-tivity to be defined later on. For the statement, recall that any compact, open subgroup~$V$ at which the minimum in~\eqref{eq:scale} is attained is said to be \emph{minimizing for~$x$}.
\begin{prop}
\label{prop:scalar=>tidy}
Let the semigroup $S$ be multiplicative over $V$. Then $V$ is minimizing for every $x\in S$ and $s(x) = s_V(x)$.
\end{prop}
\begin{proof}
Let $x\in \mathcal{S}$. Then, by the ``spectral radius formula" for the scale,  
see~\cite[Theorem~7.7]{roggi},
$$
s(x) = \lim_n [x^nVx^{-n} : x^nVx^{-n}\cap V]^{1/n} = \lim_n s_V(x^n)^{1/n}.
$$
Since $\mathcal{S}$ is multiplicative over $V$, the last is equal to $s_V(x)$
and it follows that $V$ is minimizing for $x$. 
\end{proof} 

If $x$ and $x^{-1}$ both belong to a semigroup multiplicative over $V$, then 
$$
s_V(x)s_V(x^{-1}) =s_V(xx^{-1})= s_V(e_G) = 1\hbox{ for every }x\in G.
$$ 
Hence $s_V(x) = s_V(x^{-1}) = 1$ and $xVx^{-1} = V$. The `only if' direction of the following corollary is thus established. The `if' direction is straightforward. 
\begin{corollary}
\label{cor:scalar=>tidy}
The subgroup $H\leqslant G$ is multiplicative over the compact, open subgroup $V$ if and only if $H$ normalizes $V$.
\endproof
\end{corollary}

It is shown next  that each subgroup $V$ that is minimizing for $x$ 
gives rise to a semigroup multiplicative over $V$ and containing $x$. 
An essential part of the proof is the fact that 
a compact, open subgroup $V$ that is minimizing for~$x$ 
has a decomposition 
$V = V_+V_- = V_-V_+$ where $V_\pm$ are closed subgroups of $V$ with $xV_+x^{-1}\geqslant V_+$ and $xV_-x^{-1} \leqslant V_-$.
This follows from the fact that every minimizing subgroup 
is also \emph{tidy} for~$x$, see~\cite[Definition~p.343]{Willis94} and~\cite[Theorem~3.1]{Willis01}.

\begin{lemma}
\label{lem:scale on double cosets}
Suppose that $V$ is a compact, open subgroup of $G$ that is minimizing for $x\in G$ and let $n\in \ZZ^+$. 
Then $s(y)=s(x)^n$ for every $y\in (VxV)^n$.
\end{lemma}
\begin{proof}
For $y\in (VxV)^n$ we have
$$
y\in \underbrace{(VxV)\dots (VxV)}_n = V(xVx^{-1})\dots (x^{n-1}Vx^{1-n})x^nV.
$$
The comment preceding the statement of the lemma shows  that
$$
V(xVx^{-1}) = V_-V_+(xV_+x^{-1})(xV_-x^{-1}) = V_-(xV_+x^{-1})(xV_-x^{-1}),
$$
where the last is equal to $V_-(xV_-x^{-1}) (xV_+x^{-1})$, and which is in turn equal to $V_-(xV_+x^{-1})$ because $xV_-x^{-1}\leqslant V_-$. Iterating this argument yields 
$$
V(xVx^{-1})\dots (x^{n-1}Vx^{1-n})x^nV = V_-(x^{n-1}V_+x^{1-n})x^n V.
$$
Next, since $(x^{n-1}V_+x^{1-n})x^n =  x^n(x^{-1}V_+x)\leqslant x^n V_+$, it follows that
$$
V(xVx^{-1})\dots (x^{n-1}Vx^{1-n})x^nV = V_-x^n V.
$$
Finally, using again that $V = V_-V_+$ and that $x^nV_-x^{-n}\leqslant V_-$, we obtain
\begin{equation}
\label{eq:VxV}
V(xVx^{-1})\dots (x^{n-1}Vx^{1-n})x^nV = V_-x^n V_+.
\end{equation}
Hence $V$ is tidy and therefore minimizing for $y$, by~\cite[Theorem~3]{Willis94}, and to prove the lemma it suffices to show that $[yVy^{-1} : yVy^{-1}\cap V] = s(x)^n$.

For this, it follows from \eqref{eq:VxV} that $y = t_-x^nt_+$, where $t_-\in V_-$ and $t_+\in V_+$. Then
$$
[yVy^{-1} : yVy^{-1}\cap V] = [t_-x^nV(t_-x^n)^{-1} : t_-x^nV(t_-x^n)^{-1}\cap V]
$$
because $t_+Vt_+^{-1} = V$. Conjugating by $t_-^{-1}$ shows that the index on the right is equal to 
$$
[x^nVx^{-n} : x^nVx^{-n}\cap t_-^{-1}Vt_-], \text{ which equals } 
[x^nVx^{-n} : x^nVx^{-n}\cap V]
$$ 
because $t_-\in V$. The latter is equal to $s(x^n) = s(x)^n$ because $V$ is minimizing for~$x$, whence 
$[yVy^{-1} : yVy^{-1}\cap V] = s(x)^n$
as required.
\end{proof}

The previous lemma and continuity of the scale function \cite[Corollary~4]{Willis94} imply: 

\begin{prop}
\label{prop:tidy<->scalar}
Let $x\in G$ and suppose that $V$ is a compact, open subgroup of $G$ that is minimizing for $x$. Then the semigroup generated by $x$ and $V$ is multiplicative over $V$. If $s(x)\neq 1$, then this semigroup  is not compact.
\end{prop}
\begin{proof}
To establish 
the first claim, 
note that 
the semigroup 
generated by~$x$ and~$V$ 
is $\mathcal{S}=V\cup\bigcup_{n\in\ZZ^+} (VxV)^n$.  
Given $y$ and~$z$ 
in~$\mathcal{S}$, 
consider two cases. 

In the first case, 
assume that 
$y\in (VxV)^m$ and $z\in (VxV)^n$. 
Then  
$yz$ belongs to $(VxV)^{m+n}$ 
and 
the claim follows 
because $s(y) =s(x)^m$,  $s(z)=s(x)^n$ and $s(yz)=s(x)^{m+n}$ by Lemma~\ref{lem:scale on double cosets}. 

In the second case, 
at least one of
the elements~$y$ and~$z$ belongs to~$V$ and it may be  
supposed without loss 
that~$y\in V$.  
Then $s(yz)=s(z)$ by~\cite[Theorem~3]{Willis94} and,
since $s(V)=\{1\}$, 
it follows that 
$s(yz)=s(y)s(z)$. 

For the second claim note that, if $s(x)\ne1$, then $\{x^n \colon n\in{\mathbb N}\}$ is not contained in any compact subset of $G$ by continuity of the scale function.
\end{proof}

The previous discussion shows that there are open subsemigroups of $G$ satisfying the next definition. 
\begin{definition}
\label{defn:generalmult}
Let $G$ be a totally disconnected locally compact group. A semigroup, $\mathcal{S}\subseteq G$ is \emph{scale-multiplicative}, or \emph{\scalar}, if 
$$
s(xy) = s(x)s(y)\text{ for every }x,\,y\in \mathcal{S}.
$$
\end{definition}

As remarked in the introduction, every \scalar\ semigroup is contained in a maximal one. Since the scale function is continuous the following is immediate.
\begin{prop}
\label{prop:maximal} Every \scalar\  subsemigroup of $G$ is contained in a maximal such semigroup. All maximal \scalar\ semigroups are closed. 
\endproof
\end{prop}

In view of Proposition~\ref{prop:tidy<->scalar} and Theorem~\ref{thm:class(scalar-semigroups)} below, maximal open \scalar\ semigroups are of particular interest. Should $\mathcal{S}$ contain an open subgroup~$V$, then $\mathcal{S} = \mathcal{S}V$ is open. Any semigroup containing a semigroup multiplicative over $V$ is therefore  open. In particular, any maximal \scalar\ semigroup containing a semigroup as in Proposition~\ref{prop:tidy<->scalar} is open. 
So far as we know, this is the only circumstance in which open \scalar\ semigroups occur.
\begin{question}
\label{qu:open}
Let $\mathcal{S}$ be a maximal \scalar\ semigroup that contains an open semigroup. Does $\mathcal{S}$ contain an open subgroup of $G$? 
\end{question}
Note that, even when an \scalar\ semigroup $\mathcal{S}$ contains an open subgroup, that subgroup need not be minimizing for all elements of $\mathcal{S}$. In particular, the semigroups in Theorem~\ref{thm:class(scalar-semigroups)} contain a largest open subgroup but that subgroup is not minimizing for any element that does not belong to it.

The next result implies that there is a natural involution on the set of maximal \scalar\ semigroups. 
\begin{prop}
\label{prop:involution}
The semigroup $\mathcal{S}$ is \scalar\ if and only if $\mathcal{S}^{-1}$ is \scalar.
Moreover, $\mathcal{S}$ is maximal \scalar\ if and only if $\mathcal{S}^{-1}$ is maximal \scalar.
\end{prop}
\begin{proof}
Suppose that $\mathcal{S}^{-1}$ is \scalar\ and consider $x,y\in \mathcal{S}$. Denoting the modular function on $G$ by $\Delta : G \to (\mathbb{R}^+,\times)$,~\cite[Corollary~1]{Willis94} implies that
$$
s(xy) = \Delta(xy)s((xy)^{-1}) = \Delta(x)\Delta(y) s(y^{-1})s(x^{-1}),
$$
because $\Delta$ is a homomorphism and $\mathcal{S}^{-1}$ is \scalar. Since $s(x) = \Delta(x)s(x^{-1})$ and similarly for $s(y)$, it follows that $s(xy) = s(x)s(y)$ and that $\mathcal{S}$ is a \scalar\ semigroup.

For the maximality statement, 
note that the inverse map is an involution preserving proper containment. 
\end{proof}

For any \scalar\  semigroup $\mathcal{S}$, the map $s : \mathcal{S}\to (\ZZ^+,\times)$ is a homomorphism whose range is a subsemigroup of $(\ZZ^+,\times)$. This semigroup is not equal to all of $\ZZ^+$ if $G$ is compactly generated, because the scale has only finitely many prime divisors in that case, see~\cite{Wi:finitefactors}.  


\section{Groups supporting positive scale}
\label{sec:+scale}

Let $H\rtimes_\alpha \ZZ$ be a totally disconnected, locally compact group, where $\alpha$ is an automorphism of $H$. 
By identifying the internal and external semi-direct products, we write  $s(\alpha)=s(e_H,1)$.
Suppose that there is a compact, open subgroup $V\leqslant H$ with 
\begin{equation}
\label{eq:semidirect}
\alpha(V)>V\ \text{ and }\ \bigcup_{n\in\ZZ} \alpha^n(V) = H.
\end{equation}
Then $s(\alpha) = [\alpha(V) : V]$ and  
$s(\alpha^{-1}) =[\alpha^{-1}(V) : \alpha^{-1}(V)\cap V] 
= 1$. Thus $V$ is minimizing for $\alpha^{-1}$ and hence, by~\cite[Corollary~3.11]{Willis01} or~\cite[Corollary~5.3]{roggi}, for~$\alpha$. 

Such semidirect products are of particular interest 
because every group $G$ that is not uniscalar has subgroups of this form. 
If $s(x)>1$ and $V$ is tidy for $x$, 
then $V_{++} := \bigcup_{n\in\ZZ} x^nV_+x^{-n}$ 
is a closed subgroup of $G$ that is normalized by $x$, 
and $\langle x, V_{++}\rangle=V_{++}\rtimes \langle x \rangle$, 
see~\cite[Proposition~2]{Willis94}. 
Denote the inner automorphism $y\mapsto xyx^{-1}$ of $G$ by $\alpha_x$,
so that $V_{++}\rtimes \langle x \rangle=V_{++}\rtimes_{\alpha_x} \langle x \rangle$. 
Multiplicative semigroups contained 
in $V_{++}\rtimes_{\alpha_x} \langle x \rangle$ 
will extend to maximal \scalar\ semigroups in $G$. 
By way of illustration of Proposition~\ref{prop:tidy<->scalar}, 
observe that the semigroup in~$H$ generated 
by $V$ and $x$, respectively $x^{-1}$, is the disjoint union 
$$
\bigsqcup_{n\in\NN} \alpha_{x}^{n}(V)x^n\ \text{, respectively }\ \bigsqcup_{n\in\NN} Vx^{-n}.
$$
These semigroups are multiplicative over $V$ but are not maximal. 
\begin{prop}
\label{prop:semidirect}
Let $G = H\rtimes_\alpha \ZZ$ with $V\leqslant H$ as in \eqref{eq:semidirect}. Then 
$$
\mathcal{S}_+ = \left\{ (h,n)\in H\rtimes_\alpha \ZZ \colon n\geq0\right\}\text{ and }\mathcal{S}_- = \left\{ (h,n)\in H\rtimes_\alpha \ZZ \colon n\leq0\right\}
$$ 
are the only maximal \scalar\ semigroups in $G$. These semigroups are multiplicative over $V$ if and only if $V\triangleleft H$. 
\end{prop}
\begin{proof}
Put $q:=[\alpha(V) : V] >1$. 
We derive a formula for the scale of 
$(h,n)$ in $H\rtimes_\alpha \ZZ$  in terms of $q$.
For this, choose~$m$ 
such that 
$h\in\alpha^{m}(V)$. 
Then 
$$
(h,n)\alpha^{m-n}(V)(h,n)^{-1} = (h,0)\alpha^m(V)(h,0)^{-1} = \alpha^m(V).
$$ 
Hence, if $n\geq0$, then $(h,n)\alpha^{m-n}(V)(h,n)^{-1} \geq \alpha^{m-n}(V)$ and it follows that $\alpha^{m-n}(V)$ is minimizing for~$(h,n)$ and 
$$
s(h,n) = [\alpha^m(V) : \alpha^{m-n}(V)] = q^n.
$$
If $n\leq0$, then $(h,n)\alpha^{m-n}(V)(h,n)^{-1} \leqslant \alpha^{m-n}(V)$ and $s(h,n) = 1$. Therefore
$$
s(h,n) =  \begin{cases}
q^n, & \text{ if }n\geq0\\
1, & \text{ if }n<0
\end{cases}.
$$

Since the product of $(g,m)$ and $(h,n)$ in $G$ is $(g,m)(h,n) = (g\alpha^m(h), m+n)$, 
and 
$q>1$,  
it follows that 
\begin{align}
\label{eq: decreasing}
s((g,m)(h,n)) < s(g,m)s(h,n) &\text{ if $m$ and $n$ have opposite signs, and  } \\
\label{eq: multiplicative}
s((g,m)(h,n)) = s(g,m)s(h,n) &\text{ if $m$ and $n$ have the same sign. }
\end{align}
By~(\ref{eq: multiplicative}), $\mathcal{S}_+$ and $\mathcal{S}_-$ are \scalar\ semigroups.
Since $G = \mathcal{S}_+\cup \mathcal{S}_-$,~(\ref{eq: decreasing}) implies that every \scalar\ semigroup is contained in one of  $\mathcal{S}_+$ or $\mathcal{S}_-$. It follows that 
both $\mathcal{S}_+$ and $\mathcal{S}_-$ are maximal \scalar\ semigroups, and that $G$ contains no others.

We now prove that $\mathcal{S}_+$ and $\mathcal{S}_-$ are multiplicative over $V$ if and only if $V\triangleleft H$.
Assume that $\mathcal{S}_+$ is multiplicative over $V$. Then, in particular, the subgroup $H$ is multiplicative over $V$ and $H$ normalizes $V$, by Corollary~\ref{cor:scalar=>tidy}. That $V\triangleleft H$ if $\mathcal{S}_-$ is multiplicative over $V$ follows similarly. On the other hand, if $V\triangleleft H$, then $(h,n)V(h,n)^{-1} = \alpha^n(V)$ for every $(h,n)\in G$ and it follows that 
\begin{align}
 s_V(h,n) &= [(h,n)V(h,n)^{-1} : (h,n)V(h,n)^{-1}\cap V] \nonumber
= [\alpha^n(V) : \alpha^n(V)\cap V]   \label{eq: s_v=s}\\
 &=  s(h,n)
\end{align}
for every $(h,n)\in H\rtimes_\alpha \ZZ$. 
Then
$s_V$ is multiplicative on
$\mathcal{S}_+$ and $\mathcal{S}_-$ because $s$ is, and $\mathcal{S}_+$ and $\mathcal{S}_-$ are
multiplicative over~$V$.  
\end{proof}

As previously remarked, every group $G$ that is not uniscalar has closed subgroups isomorphic to some $H\rtimes_\alpha\ZZ$. Multiplicative semigroups in each of these subgroups extend to maximal \scalar\ semigroups in $G$. 
We recall 
in the next section 
a canonical action of 
$H\rtimes_\alpha\ZZ$  
on a homogenous tree.  
The structure space of maximal \scalar\ semigroups will consequently be related to this tree.  
It will be seen also that this relationship is not straightforward even when $G$ itself acts on a tree. 

\section{Highly transitive automorphism groups of locally finite trees} 

In this section 
we determine 
the maximal \scalar\ semigroups 
for groups of automorphisms of a tree, $T$, that act 2-transitively on the boundary $\partial T$. Such groups are discussed by F.~Choucroun in~\cite{TAct.harm.c}. The standing assumption in this section is that the group~$G$ has this property, and it will be seen in Theorem~\ref{thm:class(scalar-semigroups)} that the maximal \scalar\ semigroups are closely identified with features of $T$. 
%
%
%
%
%

%
%

%
Automorphisms 
of a tree are classified 
into elliptic and hyperbolic types, see~\cite{arbre}. 
Terminology and results relating to 
this classification to be used below will be briefly 
reviewed before proceeding. The vertices of $T$ will be denoted by~$V(T)$ and its edges by~$E(T)$. We adopt the convention that the edge between vertices $v$ and $w$ comprises the ordered pairs $(v,w)$ and $(w,v)$. Then $(v,w)$ is the \emph{oriented edge} from $v$ to $w$ and $v$ is 
the \emph{initial vertex 
of the edge\/} 
while $w$
is the 
\emph{the terminal vertex\/}. If~$e$ is an oriented edge, then $o(e)$ and $t(e)$ denote its initial and terminal vertices respectively, and $\bar{e}$ is the oppositely oriented edge with initial vertex~$t(e)$ and terminal vertex~$o(e)$. 

%

%
Every automorphism
of a tree 
either 
fixes a vertex, 
inverts an edge 
or induces a 
nontrivial translation 
along an infinite geodesic
in the tree. 
The first two types 
are called \emph{elliptic}
and have at least one fixed point in the tree.%
\footnote{Note that, 
in contrast with 
 some parts of 
the literature,
we consider tree automorphisms  
that invert edges, 
so-called \emph{inversions\/},  
to be elliptic. }
The last type is 
\emph{hyperbolic}, 
in which case 
the geodesic of translation
is unique
and called 
the \emph{axis\/}. 
Define the 
length,~$\ell(g)$, of
an automorphism $g$ of the tree to be 
the minimum distance, in the usual graph metric, 
by which it moves 
points in the tree. 
The length of $g$ is $0$ for $g$ elliptic, 
and is the length of translation
along the axis for $g$ hyperbolic.

%
%
Define 
the \emph{minimal set for $g$} to be 
\begin{equation*}
\min(g):=\{p\in T\colon d(p,g.p)=\ell(g)\}.
\end{equation*}
Thus 
$\min(g)$ is 
the set of fixed points 
of~$g$ 
if~$g$ is elliptic 
and 
the axis
of~$g$ 
if~$g$ is hyperbolic. 

Two oriented edges 
in a tree are 
\emph{coherent\/} if 
the distance 
between their respective initial vertices 
equals 
the distance 
between their terminal vertices 
and 
\emph{incoherent\/} 
otherwise. 
%
A hyperbolic automorphism, 
$h$ say,  
is said to 
\emph{translate 
along\/} 
$(o(e),t(e))$ 
if and only if 
the oriented edges~$(o(e),t(e))$ and~$(h.o(e),h.t(e))$ 
are 
coherent and $d(o(e),h.t(e)) = d(t(e),h.o(e))+ 2$. 


%
A semi-infinite geodesic in a tree is a \emph{ray}.
Two rays 
in a tree 
are said to 
belong to 
\emph{the same end\/}, 
if and only if 
their intersection 
is a ray. 
This defines 
an equivalence relation 
on the set of rays 
in the tree, 
whose equivalence classes 
are called 
\emph{the set of ends 
of the tree\/}. 
The automorphism group 
of the tree 
acts 
on 
its set of ends. 
Given a hyperbolic automorphism, 
$h$ say, 
each ray $\mathfrak{r}\subseteq\min(h)$ satisfies either 
$h.\mathfrak{r}\subset \mathfrak{r}$ or $h.\mathfrak{r}\supset \mathfrak{r}$ and these conditions distingish  
two distinct ends of the axis of $h$.  The first is called 
\emph{the attracting end of~$h$} 
and denoted by~$\epsilon_+(h)$, while the second is called
\emph{the repelling end of~$h$} 
and denoted by~$\epsilon_-(h)$

We introduce some notation for segments. Given two vertices $v$ and $w$, by the \emph{segment} $]v,w]$ we mean the set of vertices on the path from $v$ to $w$ not including~$v$.  
By Lemma~31 from~\cite{direction(aut(tdlcG))},
the scale of a hyperbolic element $h$ is 
the product 
\begin{equation}\label{scale of segment}
\prod_{v\in ]w,h.w]}q(v),
\end{equation}
where $q(v)$ is the valency of the vertex $v$ minus 1
in the minimal invariant subtree of $G_{\epsilon_+(h)}$, and $w$ is any vertex on the axis of $h$.

Having recalled 
basic results 
for groups acting on trees 
we now 
briefly 
return to 
the groups $H\rtimes_\alpha\ZZ$ 
considered in 
the last section, 
whose basic importance 
for \scalar\ semigroups 
was explained there. 

The group $H\rtimes_\alpha\ZZ$ 
is an HNN-extension and consequently acts on a homogeneous tree $T_{q+1}$ with valency $q+1$ and fixes an end $\omega$, see~\cite[Section~4]{contraction(aut(tdlcG))}. 
Denote the corresponding representation by $\rho\colon H\rtimes_\alpha\ZZ \to \text{\rm Aut}(T_{q+1})$. The kernel of $\rho$ is the largest compact, normal subgroup of $H\rtimes_\alpha \ZZ$, the group $\rho(H)$ is contained in the set of elliptic automorphisms of the tree, and $\rho(\alpha)$ is a hyperbolic element that has $\omega$ as its attracting end. 
Then 
the maximal \scalar\ semigroups 
in $H\rtimes_\alpha\ZZ$  
have the following characterization. 
\begin{align*}
\mathcal{S}_+ &= \left\{ x\in H\rtimes_\alpha\ZZ \colon \rho(x)\text{ is elliptic or } \rho(x) \text{ has }\omega\text{ as its attracting end}\right\}\\
\intertext{ 
and }
\mathcal{S}_- &= \left\{ x\in H\rtimes_\alpha\ZZ \colon \rho(x)\text{ is elliptic or } \rho(x) \text{ has }\omega\text{ as its repelling end}\right\}.
\end{align*}
Denote the axis of $\rho(\alpha)$ by $\mathfrak{l}$. 
Then $\rho(V)$ is the fixator of a ray,  $[v_0,\omega)$ say, on $\mathfrak{l}$. 
Label the other vertices on $\mathfrak{l}$ as $v_n = \rho(\alpha^n).v_0$. 
Then the subset 
$$
(\alpha^{n-1}(V), n) := \left\{ (x,n) \colon x\in \alpha^{n-1}(V)\right\} \subseteq H\rtimes_\alpha\ZZ
$$ 
of the semigroup generated by $V$ and $\alpha$ translates $v_{j}$ to $v_{j+n}$ for every $j\geq-1$, and the subset $(V,{-n})$ of the semigroup generated by $V$ and $\alpha^{-1}$ translates $v_j$ to $v_{j-n}$ for every $j\geq n$. 

This explains 
how maximal \scalar\ semigroups 
are represented geometrically 
for these examples. 
We next return to 
groups with a 2-transitive action 
on the set of ends 
of a tree. 

The following lemma
will be used 
on numerous occasions 
in this section. 
It appears in~\cite[Lemma~1.2]{l.rank1} 
and~\cite[Lemma~6.8]{CovGG},  
and 
implicitly in~\cite[Lemme~3.1]{arbre}.

\begin{lemma}
\label{lem:detect-axis}
Suppose that 
an oriented edge, $e$,  
and its image 
under an automorphism~$g$ 
differ 
and 
are coherent. 
Then 
$g$ is hyperbolic 
and the edges 
$e$ and $g.e$  
lie 
on the axis 
of~$g$.  
\endproof
\end{lemma}

Proposition~3.4 in~\cite{arbre} 
states that 
every group of elliptic automorphisms 
of a tree 
fixes 
either 
a vertex, 
inverts an edge,  
or 
fixes 
an end 
of the tree. 
The  
arguments to follow require 
a version of this result 
for semigroups. 
One way to establish this is to work through the proof in \cite[Proposition~7.2]{CovGG}, 
which uses~\cite{Lambda-trees}, 
and observe that it also applies to semigroups. 
We take the alternative approach of deducing 
it from the corresponding statement for groups.

\begin{lemma}
\label{lem:struc(elliptic_semigroups)}
Every semigroup of elliptic automorphisms 
of a tree 
fixes 
either 
a vertex, 
inverts an edge,  
or 
fixes 
an end 
of the tree. 
%
\end{lemma}
\begin{proof}
The set of hyperbolic automorphisms of the tree is open because, given a hyperbolic automorphism $x$, any automorphism that agrees with $x$ on two adjacent vertices on its axis is also hyperbolic, by Lemma~\ref{lem:detect-axis}. Hence the set of elliptic automorphisms of the tree is closed and it may be assumed that the given semigroup is closed. Since the closed semigroup generated by a single elliptic automorphism is compact, it is a group. Hence the given closed semigroup is in fact a group and Proposition~3.4 in~\cite{arbre} applies. 
\end{proof}

The above proof in fact shows more than is asserted in the lemma. 
\begin{corollary}
\label{cor:struc(elliptic_semigroups)}
Every closed semigroup of elliptic automorphisms of a tree is in fact a group. If this group does not fix an end of the tree, then it is compact.
\endproof
\end{corollary} 

Elliptic elements 
normalize 
the stabilizer of 
a point of the tree, 
which is 
a compact, open subgroup. 
Hence the scale function of 
a closed subgroup, $G$, of 
the automorphism group of 
a locally finite tree 
takes the value~$1$ 
at each elliptic element. 
Although hyperbolic elements 
in~$G$ 
can also have scale~$1$, this does not occur under the  hypothesis that $G$ acts $2$-transitively on $\partial T$, as the next result shows.  

\begin{prop}
\label{prop:2-trans_implies_not_uniscalar}
Let $T$ be a tree that has no leaves and for which there is an edge~$\{e,\bar{e}\}$ such that each component of $T\setminus \{e,\bar{e}\}$ has at least two ends. Let $G$ be a group of automorphisms of~$T$ acting 2-transitively on $\partial T$. Then:
\begin{enumerate}
\item for every distinct pair $\omega_1, \omega_2\in \partial T$ the geodesic $]\omega_1,\omega_2[$ in $T$ is the axis of some hyperbolic element in $G$ and, moreover, the hyperbolic elements may be chosen from a single conjugacy class; 
\label{prop:2-trans_implies_not_uniscalar1}
\item if $v$ is a vertex with valency greater than~$2$, then $G_v$ acts $2$-transitively on the sphere with centre~$v$ and radius~$r$ for every~$r\geq1$;
\label{prop:2-trans_implies_not_uniscalar1a}
\item the group $G$ has at most two orbits, say $O$ and $E$, on the set of vertices of $T$ that have valency greater than 2; denoting by $k$ the minimal distance between distinct elements of $O\cup E$, every vertex in $T$ is within $k$ of $O$ and~$E$;
\label{prop:2-trans_implies_not_uniscalar2}
\item the vertices of every geodesic $]\omega_1,\omega_2[ = (\dots, v_j, v_{j+1}, \dots)$ in $T$ satisfy that $v_j\in E$ for some $j\in\ZZ$ and that: $v_{j+nk}\in E$ for all even~$n$; $v_{j+nk}\in O$ for all odd~$n$; and $v_i$ has valency~$2$ otherwise; 
\label{prop:2-trans_implies_not_uniscalar2a}
\label{prop:2-trans_implies_not_uniscalar3}
\item every hyperbolic element $h\in G$ translates $]\epsilon_-(h),\epsilon_+(h)[$ through a distance~$nk$ for some $n\in\ZZ$  
and, if $G$ is closed in $\text{\rm Aut}(T)$, the scale function of $G$ is given by 
\begin{equation}
\label{eq:scale_calc}
s(g) =(q_Oq_E)^{\ell(g)/2k}
\end{equation}
where $q_E+1$ and $q_O+1$ are the valencies of vertices in $E$ and $O$ respectively;
\label{prop:2-trans_implies_not_uniscalar4}
\item for each end $\omega\in \partial T$ the minimal invariant subtree of the subgroup $G_\omega$ is the whole tree.
\label{prop:2-trans_implies_not_uniscalar6}
\end{enumerate}
\end{prop}
\begin{proof}
\eqref{prop:2-trans_implies_not_uniscalar1} By  \cite[Proposition~7.2]{CovGG}, if $G$ does not contain any hyperbolic element, then it fixes a vertex, inverts an edge or fixes an end of~$T$. The last is ruled out by the 2-transitivity of $G$ on $\partial T$. It may be seen that $G$ cannot fix a vertex either. For this, assume that $v$ is fixed by $G$ and let $\{f,\bar{f}\}$ be the edge adjacent to $v$ that is closest to the edge $\{e,\bar{e}\}$ in the statement of the proposition. Then the component of $T\setminus\{f,\bar{f}\}$ containing $v$ has at least one end,~$\omega$, because $T$ has no leaves, and the component of $T\setminus\{f,\bar{f}\}$ to which~$v$ does not belong contains a component of $T\setminus\{e,\bar{e}\}$ and has at least two ends,~$\omega_1$ and~$\omega_2$. Then no element of $G$ can fix $\omega_1$ and move $\omega$ to $\omega_2$, which contradicts the $2$-transitivity of $G$ on $\partial T$. Similarly, assuming that $G$ inverts an edge, $\{f,\bar{f}\}$, at least one component of $T\setminus \{f,\bar{f}\}$ has two or more ends $\omega_1$ and $\omega_2$ and the other has at least one end~$\omega$. No element of $G$ can fix $\omega_1$ and send $\omega$ to $\omega_2$, which again contradicts the $2$-transitivity of~$G$. Assuming that~$G$ contains only elliptic elements thus leads to a contradiction. Hence $G$ contains a hyperbolic element, $h$ say. 

Since $G$ is $2$-transitive on $\partial T$, there is $g\in G$ that sends $\epsilon_+(h)$ to $\omega_1$ and $\epsilon_-(h)$ to $\omega_2$. Then $ghg^{-1}$ is hyperbolic and $]\omega_1,\omega_2[$ is its axis. 

\eqref{prop:2-trans_implies_not_uniscalar1a} Let $w_1$ and $w_2$ be two vertices on the sphere with centre~$v$ and radius~$r$. Then the `hyperbolic triangle' with vertices $v$, $w_1$ and $w_2$ contains at most two edges incident on~$v$ and so there is an edge, $\{e,\bar{e}\}$ say, incident on~$v$ and not in this triangle. Choose rays $\mathfrak{r}$, $\mathfrak{r}_1$ and $\mathfrak{r}_2$ originating at $v$ and passing through $e$, $w_1$ and $w_2$ respectively, and let $\omega$, $\omega_1$ and $\omega_2$ be the corresponding ends. Since $G$ acts $2$-transitively on $\partial T$, there is $g\in G$ such that $g.\omega = \omega$ and $g.\omega_1 = \omega_2$. Then $g\in G_v$ and $g.w_1 = w_2$. Hence $G_v$ acts $2$-transitively on the sphere as claimed.

\eqref{prop:2-trans_implies_not_uniscalar2}
Let $w_0\ne v_0$ be two vertices with valency greater than~$2$. That there is at least one such vertex, $v_0$ say, follows because $T$ has at least three ends, and that there is a second follows because every vertex in $G.v_0$ has this property. Then every vertex on the sphere with centre $v_0$ and radius $d(v_0,w_0)$ belongs to the orbit $G_{v_0}.w_0$ and every vertex on the sphere with centre $w_0$ and radius $d(v_0,w_0)$ belongs to the orbit $G_{w_0}.v_0$. It follows that all spheres with centre $v_0$ and radius $2nd(v_0,w_0)$, $n\geq1$, are contained in the orbit $G.v_0$. Therefore every vertex is within distance $d(v_0,w_0)$ of $G.v_0$. 

Choose~$v_0$ and~$w_0$ above so that $d(v_0,w_0)$ is minimized and let~ $k$ be this minimum value. Let~$E$ be the set of vertices at distance~$nk$ from~$v_0$ with~$n$ even and~$O$ be the set of vertices at distance~$nk$ with~$n$ odd. Then, as seen above, every vertex is within distance~$k$ of $E$ and $O$, and $E\subseteq G.v_0$ and $O\subseteq G.w_0$. 

\eqref{prop:2-trans_implies_not_uniscalar3}  It will be shown first that there is $j\in\ZZ$ such that $v_j$ belongs to~$E$. Given $u\in ]\omega_1,\omega_2[$, there is $w'\in O$ within distance~$k$. Hence~$u$ is on or inside the sphere with centre~$w'$ and radius~$k$. This sphere is contained in~$E$ and $]\omega_1,\omega_2[$ must pass through it in order to escape to $\omega_1$. Therefore there is a vertex of $]\omega_1,\omega_2[$ that belongs to~$E$. All other claims follow immediately.

\eqref{prop:2-trans_implies_not_uniscalar4}  The claim about translation length follows from the description of geodesics in~\eqref{prop:2-trans_implies_not_uniscalar3}. Since each of~$E$ and $O$ is contained in a single $G$-orbit, all vertices in these sets have the same valency, which is denoted $q_E+1$ and $q_O+1$ as stated. If~$h$ translates through~$nk$ with~$n$ odd, then $E$ and $O$ are both contained in a single $G$-orbit and $q_E = q_O = q$. In this case each segment $]w,h.w]$, with~$w$ on the axis of~$h$, contains~$n$ vertices with valency~$q$ and~\eqref{scale of segment} yields the claimed value of $s(h)$. If~$h$ translates through~$nk$ with~$n$ even, then~$q_E$ and~$q_O$ may differ but each segment $]w,h.w]$, with~$w$ on the axis of~$h$, contains~$n/2$ vertices with each valency~$q_E$ and~$q_O$ and~\eqref{scale of segment} applies again. 

\eqref{prop:2-trans_implies_not_uniscalar6} Every edge $\{e,\bar{e}\}$ lies on a geodesic $]\omega,\omega'[$ for some $\omega'\in\partial T$. Hence,  by~\eqref{prop:2-trans_implies_not_uniscalar1}, $\{e,\bar{e}\}$ lies on the axis of a hyperbolic element in $G_\omega$. Every such axis is contained in the minimal tree invariant under $G_\omega$. 
\end{proof}

\begin{remark}
It has in fact been seen in the proof of Proposition~\ref{prop:2-trans_implies_not_uniscalar} that the existence of a group of automorphisms of~$T$ acting $2$-transitively on $\partial T$ implies, under a weak hypothesis on~$T$ otherwise, that~$T$ may be obtained from a semi-homogeneous tree by subdividing each edge into~$k$ edges. This conclusion is similar to that of~\cite[Th\'eor\`eme~1.6.1]{TAct.harm.c}. 

Therefore for the remainder of this section the tree~$T$ may be assumed, without any loss of generality on the hypotheses of Proposition~\ref{prop:2-trans_implies_not_uniscalar}, to be semi-homogeneous with all vertices having valency at least~$3$.
\end{remark}

It follows immediately from Proposition~\ref{prop:2-trans_implies_not_uniscalar}\eqref{prop:2-trans_implies_not_uniscalar3} that hyperbolic elements~$g$ and~$h$ satisfy $s(gh) = s(g)s(h)$ if and only if their translation lengths add, that is, $\ell(gh) = \ell(g)+\ell(h)$, and similarly if the scale of the product is less than (or greater than) the product of the scales. The next few results, which apply to all trees, describe the circumstances in which translation lengths of hyperbolic elements add.

\begin{lemma}\label{lem:scale(prod(hyperbolics no edge common))}
Let~$t_1$ and~$t_2$ 
be 
hyperbolic automorphisms 
of a tree whose axes
do not have 
any edge 
in common.
Let 
$\mathfrak{d}$ 
the unique shortest path 
connecting 
the axes of~$t_1$ and~$t_2$,
and 
let~$d$ 
be 
the number of 
edges of~$\mathfrak{d}$. 
Then the following are true.
\begin{enumerate}
\item 
The 
product~$t_2t_1$ 
is hyperbolic 
and $\ell(t_2t_1)=\ell(t_2)+\ell(t_1)+2d$.
\label{lem:scale(prod(hyperbolics no edge common))1}
\item 
The axis of $t_2t_1$ contains 
the union of:~$t_1^{-1}(\mathfrak{d})$; 
the last~$\ell(t_1)$ edges 
preceding~$\mathfrak{d}$ 
on the axis 
of~$t_1$; 
$\mathfrak{d}$;  
the first~$\ell(t_2)$ edges 
succeeding~$\mathfrak{d}$ 
on the axis 
of~$t_2$; and~$t_2(\mathfrak{d})$.
%
\label{lem:scale(prod(hyperbolics no edge common))2}
\item The direction of translation 
of~$t_2t_1$ 
agrees with 
those of~$t_1$ and~$t_2$ 
on the common segments 
of their axes, and is in the direction from the axis of $t_1$ to that of $t_2$ on $\mathfrak{d}$. 
\label{lem:scale(prod(hyperbolics no edge common))3}
\end{enumerate}
\end{lemma}
\begin{proof}
Let $a$ be the vertex common to both $\mathfrak{d}$ and the axis of $t_1$, and $b$ the vertex common to  $\mathfrak{d}$ and the axis of $t_2$. The automorphism $t_2t_1$ maps $t_1^{-1}.\mathfrak{d}$ to $t_2.\mathfrak{d}$. Since $t_1^{-1}.a$ lies on the axis of $t_1$ and $t_2.b$ lies on the axis of $t_2$, the subtree spanned by $t_1^{-1}.\mathfrak{d}$ and $t_2.\mathfrak{d}$ is the path from $t_1^{-1}.b$ to $t_2.a$ of length $\ell(t_1) + \ell(t_2) + 3d$, as shown in~Figure~\ref{not_intersect}. 
\begin{figure}[htbp]
\begin{center}
\begin{center}
\includegraphics[width=14cm]{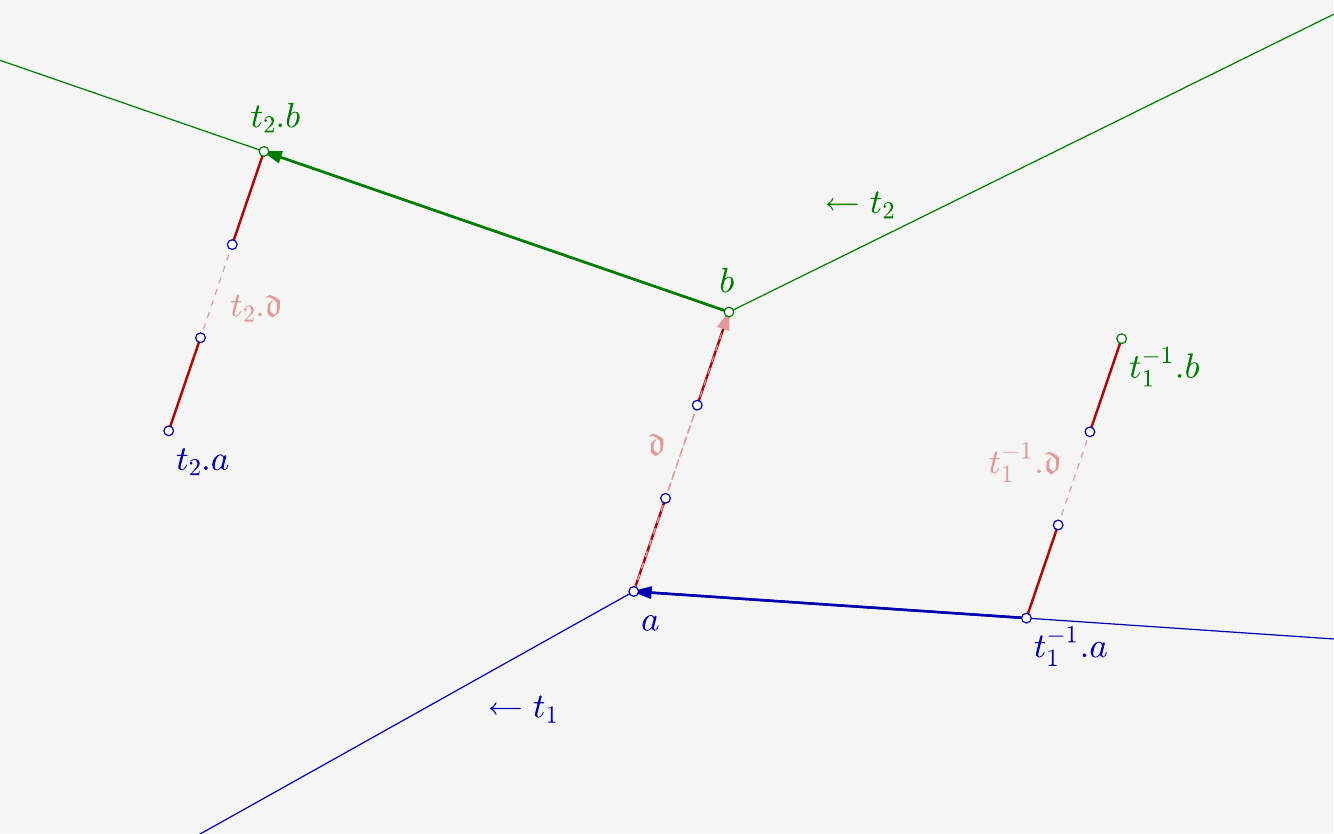}
\end{center}
\caption{{\bf The axes of $t_1$ and $t_2$ do not intersect in an edge}}
\label{not_intersect}
\end{center}
\end{figure}
Since each edge $e\in t_1^{-1}.\mathfrak{d}$ is coherent with $(t_2t_1).e\in t_2.\mathfrak{d}$,  statements~\eqref{lem:scale(prod(hyperbolics no edge common))1}--\eqref{lem:scale(prod(hyperbolics no edge common))3} follow by applying 
Lemma~\ref{lem:detect-axis} and computing the distance from $t_1^{-1}.b$ to $t_2.b$. 
\end{proof}

\begin{lemma}\label{lem:scale(prod(hyperbolics nontrivial intersection coherent))}
Let~$t_1$ and~$t_2$ 
be 
hyperbolic automorphisms 
of a tree whose axes
intersect in a nontrivial path 
$\mathfrak{c}$ 
along which~$t_1$ and~$t_2$ 
translate 
in the same direction.
Then 
\begin{enumerate}
\item  the 
product~$t_2t_1$ 
is hyperbolic
and the 
direction of translation 
of~$t_2t_1$ 
agrees with 
those of~$t_1$ and~$t_2$ 
on the common segments 
of their axes. 
\item
the axis of $t_2t_1$ contains the union
of: the edges in~$\mathfrak{c}$, 
the last~$\ell(t_1)$ edges 
preceding~$\mathfrak{c}$ 
on the axis 
of~$t_1$ 
and 
the first~$\ell(t_2)$ edges 
succeeding~$\mathfrak{c}$ 
on the axis 
of~$t_2$.
\item  
 $\ell(t_2t_1)=\ell(t_2)+\ell(t_1)$.  
\end{enumerate}
%
%
%
\end{lemma}
\begin{proof}
If $e$ is any edge on $\mathfrak{c}$, then $t_1^{-1}.e$, $e$ and $t_2.e$ are all coherent because $t_1$ and $t_2$ translate in the same direction and the coherence relation on edges is transitive. Parts~(1) --(3) then follow by applying 
Lemma~\ref{lem:detect-axis}. 
\end{proof}

\begin{lemma}\label{lem:scale(prod(hyperbolics nontrivial intersection incoherent))}
Let~$t_1$ and~$t_2$ 
be 
hyperbolic automorphisms 
of a tree whose axes
intersect in a nontrivial path 
$\mathfrak{c}$
along which~$t_1$ and~$t_2$ 
translate 
in opposite directions.
 
Then the product $t_2t_1$ may be hyperbolic or elliptic and 
$\ell(t_2t_1) < \ell(t_2)+\ell(t_1)$. 
\end{lemma}
\begin{proof}
Let 
$e = (v,w)$ be an edge on $\mathfrak{c}$ and suppose that $t_1$ translates along $(v,w)$ and $t_2$ along $(w,v)$. Then $t_2t_1$ maps the vertex $t_1^{-1}.w$ to $t_2.w$, and the path from $t_1^{-1}.w$ to $w$ and then to $t_2.w$ includes the reversal from $v$ to $w$ and back to $v$. Hence $\ell(t_2t_1)$ is at most $d(t^{-1}_1.w,v) + d(v,t_2.w) = \ell(t_1) + \ell(t_2) - 2$. 
\end{proof}

Lemmas proved up to this point allow the pairs of elements belonging to a semigroup on which translation distances add to be characterized in terms of their minimizing sets. 
\begin{lemma}\label{lem:nec-cond(mult-semiGs)}
If two automorphisms 
of a locally finite tree 
are contained 
in a semigroup on which the translation length is an additive function,  
then 
their minimal sets 
intersect non-trivially.  
\end{lemma}
\begin{proof}
If  
both automorphisms 
are elliptic the claim follows from Lemma~\ref{lem:struc(elliptic_semigroups)}. 
When both both automorphisms 
are  hyperbolic, 
it follows from Lemma~\ref{lem:scale(prod(hyperbolics no edge common))}. 

If one of the automorphisms, 
$h$ say,  
is hyperbolic 
while 
the other, 
$r$ say,  
is elliptic, 
we give  
a contrapositive 
argument 
as follows. 

By additivity, 
both $h$ and~$rh$ 
are hyperbolic  
and of the same translation length. 
Suppose the minimal sets of $r$ and $h$ have trivial interection.
Then~$r$ 
does not fix 
any point 
of the axis 
of~$h$. 
Thus 
the axes 
of~$h$ and~$rh$ 
do not intersect. 
By Lemma~\ref{lem:scale(prod(hyperbolics no edge common))}(4)
we conclude that 
$\ell(rh)>\ell(r)+\ell(h)$, 
hence 
that~$r$ and~$h$ 
do not lie in 
a common additive semigroup
and the claim 
is verified. 
\end{proof}

Certain \scalar\ semigroups 
in $G$ (acting $2$-transitively on $\partial T$) are described next in preparation for showing, in Theorem~\ref{thm:class(scalar-semigroups)}, that every such semigroup is contained in one of these. Additional notation is required for the characterization of the \scalar\ semigroups. For each $v\in V(T)$, put 
$$
E(v) = \left\{ e\in E(T) \colon t(e) = v\right\}
$$ 
be the set of edges incident on and directed to~$v$. For each proper non-empty subset, $I\subset E(v)$, put 
$$
{I}^* = \left\{ \bar{e}\colon e\in E(v)\setminus I\right\},
$$
a set of edges incident on and directed away from~$v$. We shall also write
$$
G_{(v)} = \left\{ g\in G \colon v\in \min(g)\right\},
$$
the set of elements of~$G$ that fix $v$ if they are elliptic or such that their axis passes through~$v$ if they are hyperbolic. If~$g$ is in $G_{(v)}$ and is hyperbolic, and if there are $e_1\in I$ and $e_2\in{I}^*$ on the axis of~$g$ and such that~$g$ translates in the direction of~$e_1$ and~$e_2$, we shall say that \emph{$g$ translates in through~$I$ and out through~${I}^*$}.

\begin{prop}\label{def:compatible-with-orientations}
Let $T$ be a semi-homogeneous tree in which each vertex has valency at least three. Let~$G$ be 
a group of automorphisms of the tree~$T$ that acts $2$-transitively on~$\partial T$. For~$v$ a vertex in~$T$ and $I$  a proper, non-empty subset 
of~$E(v)$ put
\begin{align}
\GIv &:= \left\{g\in G_{(v)} \colon  g.I = I \hbox{ if $g$ is elliptic, and }
\right.\\
& \phantom{AAAA}  \left.g \hbox{ translates in through $I$ and out through ${I}^*$ if $g$ is hyperbolic }\right\}\notag
\end{align}
For an end $\omega$ of the tree 
put
\begin{align}
\Gom
&:=
\left\{g\in G_\omega\colon g\hbox{ is elliptic, or }g\hbox{ is hyperbolic and }\omega = \epsilon_+(g) \right\} \hbox{ and }\\
\omG 
&:=
\left\{g\in G_\omega\colon g\hbox{ is elliptic, or }g\hbox{ is hyperbolic and }\omega = \epsilon_-(g) \right\}. 
\end{align}
Then 
$\GIv$, $\Gom$, and $\omG$ are \scalar\  subsemigroups of $G$. \\Moreover $\GIv\ne \GJv$ whenever $I\ne J$.
\end{prop}
\begin{proof}
As was seen in Proposition~\ref{prop:2-trans_implies_not_uniscalar}, elements $g,h\in G$ satisfy $s(gh) = s(g)s(h)$ if and only if $\ell(gh) = \ell(g)+\ell(h)$, and the proof below refers to lemmas about additivity of translation length for statements about multiplicativity of the scale.
 
We begin by showing that $\GIv$ is \scalar\ and closed under multiplication.
Consider $g,h\in\GIv$. 
If both $g$ and $h$ are elliptic and stabilise $I$, 
then the product will again be elliptic and stabilise~$I$, and hence back in $\GIv$. 
If both $g$ and $h$ are hyperbolic, 
translating in through $I$ and out through ${I}^*$, 
then we are either in the case of 
Lemma~\ref{lem:scale(prod(hyperbolics no edge common))} with $d=0$ 
or Lemma ~\ref{lem:scale(prod(hyperbolics nontrivial intersection coherent))}. 
Either way,  $s(gh)=s(g)s(h)$ as required.
Hence it remains only to prove that a product of 
an elliptic element and a hyperbolic element 
in $\GIv$ is again back in  $\GIv$.
Suppose $g$ is elliptic stabilizing $I$ 
and $h$ is hyperbolic translating in through $e_1\in I$ and out through $e_2\in {I}^*$.
Consider first $hg$.  Then $g^{-1}. e_1\in I$ and  $hg. (g^{-1}. e_1)=h. e_1$ is coherent with and different from $g^{-1}. e_1$ because $h$ is hyperbolic. 
Hence, by Lemma~\ref{lem:detect-axis}, $hg$ is hyperbolic and the edges $g^{-1}. e_1$ and $h. e_1$ lie on the axis of $hg$. 
The path between these two edges contains $v$ and $e_2\in {I}^*$. Hence $hg\in \GIv$. Moreover, $\ell(hg)=\ell(h)$ and hence $s(hg)=s(h)=s(h)s(g)$.
Finally, consider $gh$. A similar argument shows that $e_1$ and $gh.e_1$ are coherent and different, proving that $gh$ is hyperbolic by Lemma~\ref{lem:detect-axis}. Since  $g.e_2\in {I}^*$ is on the axis of $gh$, we conclude that $gh\in \GIv$. Moreover, $\ell(gh)=\ell(h)$ and hence $s(gh)=s(h)=s(g)s(h)$. Hence $\GIv$ is an \scalar\ subsemigroup of $G$.

For each non-empty and proper subset,~$I$, of~$E(v)$ and~$e_1\in I$ and $e_2\in I^*$, there are ends~$\omega_-$ and~$\omega_+$ of~$T$ such that~$e_1$ lies on the ray~$]\omega_1,v]$ and~$e_2$ lies on the ray $[v,\omega_2[$. By Proposition~\ref{prop:2-trans_implies_not_uniscalar}\eqref{prop:2-trans_implies_not_uniscalar1}, there is a hyperbolic element~$h\in G$ with axis $]\omega_1,\omega_2[$. It may be supposed that~$h$ translates in through~$e_1$ and out through~$e_2$, so that~$h$ belongs to~$\GIv$. Hence, for every $(e_1,e_2)\in I\times I^*$ there is~$h\in\GIv$ that translates along~$e_1$ and~$e_2$. Conversely, if~$h$ is hyperbolic with~$v$ on its axis, and~$h$ translates in through~$e_1$ and out through~$e_2$, then~$h$ does not belong to~$\GIv$ unless $(e_1,e_2)\in I\times I^*$. Therefore, if $I\ne J$ are two non-empty, proper subsets of~$E(v)$, then $\GIv\ne \GJv$.

Consider now $\Gom$. 
Given two hyperbolic elements in $\Gom$, 
their axes intersect in a ray in the equivalence class $\omega$. 
Hence they satisfy the hypotheses of 
Lemma~\ref{lem:scale(prod(hyperbolics nontrivial intersection coherent))} 
and their product is again a hyperbolic element in $\Gom$ 
and the scale of the product is the product of their scales. 
All elements in $\Gom$ fix $\omega$. 
Hence, given an elliptic element $g$ and 
a hyperbolic element $h$ in $\Gom$, 
there is some ray $\frak{r}$ in $\omega$ 
that is fixed by $g$ and translated towards $\omega$ by $h$. 
Hence $\frak{r}$ is translated by both $gh$ and $hg$, 
and $\ell(gh)=\ell(h)=\ell(hg)$ as required.

Since $\omG$=$(\Gom)^{-1}$, 
the result for $\omG$ follows from 
the result for $\Gom$ and  
Proposition~\ref{prop:involution}.
\end{proof}

We finally come to the description of the maximal \scalar\ semigroups in~$G$.

\begin{theorem}\label{thm:class(scalar-semigroups)}
Let $T$ be a semi-homogeneous tree in which each vertex has valency at least three. Suppose that~$G$ is a closed subgroup of the automorphism group of $T$ that acts $2$-transitively on $\partial T$. 
Then every \scalar\ semigroup in~$G$ is contained in one of  
the following types, which are hence maximal. 
\begin{enumerate}
\item  
The fixator of 
the midpoint 
of an edge 
that is 
inverted by~$G$.
\item 
The fixator of a vertex.
\item 
A set
of the form~$\GIv$ 
as defined in 
Proposition~\ref{def:compatible-with-orientations}, with $I\subset E(v)$ proper and not empty.  
\item
A set
of the form~$\Gom$ or~$\omG$ 
for an end~$\omega$ 
as defined in 
Proposition~\ref{def:compatible-with-orientations}.
\end{enumerate}
\end{theorem}
\begin{proof}
Let~$\mathcal{S}$ 
be
a \scalar\ semigroup 
in $G$.  
We will show
that~$\mathcal{S}$ 
is contained in 
one of the \scalar\ semigroups 
listed. For each subset,~$\mathcal{F}$, of $\mathcal{S}$ put $\min(\mathcal{F}) = \bigcap_{g\in\mathcal{F}} \min(g)$. Then, since $\min(g)$ is a subtree and the intersection of subtrees is a subtree,  Lemma~\ref{lem:nec-cond(mult-semiGs)} and~\cite[Lemma~10 in section~6.5]{trees} show that $\min(\mathcal{F}) \ne \emptyset$ for every finite $\mathcal{F}$. 

Suppose first that $\min(\mathcal{S}) = \emptyset$. Then for each edge $e$, there is a finite subset $\mathcal{F}\subset \mathcal{S}$ such that $\min(\mathcal{F})$ is contained in one of the semitrees obtained by deleting~$\{e,\bar{e}\}$. Otherwise, $e$ would be in $\min(\mathcal{F})$ for every $\mathcal{F}$ and therefore in $\min(\mathcal{S})$. Hence there is a unique end, $\omega$, of the tree such that every subtree  $\min(\mathcal{F})$ contains a ray belonging to $\omega$. It follows that each element of $\mathcal{S}$ is either elliptic and fixes a ray in $\omega$ or is hyperbolic and translates a ray in $\omega$. By Lemma~\ref{lem:scale(prod(hyperbolics nontrivial intersection incoherent))}, no two hyperbolic elements can translate in opposite directions and so we are in case~(4). Note that there do exist hyperbolic elements in~$G_\omega$, by Proposition~\ref{prop:2-trans_implies_not_uniscalar}\eqref{prop:2-trans_implies_not_uniscalar1}, and so $\Gom$ and $\omG$ are distinct semigroups, and the $2$-transitivity of~$G$ on~$\partial T$ ensures, by Proposition~\ref{prop:2-trans_implies_not_uniscalar}
\eqref{prop:2-trans_implies_not_uniscalar1}, that $\min(\mathcal{S})$ is indeed empty. 

Suppose next that $\mathcal{S}$ contains an edge inversion $g$. Then $\min(g)$ is a singleton, namely, the midpoint, $p$, of the edge. Since, by Lemma~\ref{lem:nec-cond(mult-semiGs)}, $\min(h)$ intersects $\{p\}$ for every $h\in \mathcal{S}$, it follows that $\{p\} = \min(\mathcal{S})$ and we are in case~(1). 

It remains to treat the case when $\min(\mathcal{S})$ is not empty and, since $\mathcal{S}$ does not invert any edge, contains a vertex~$v$. If $\mathcal{S}$ contains only elliptic elements, then~$v$ is fixed by all elements of $\mathcal{S}$ and we are in case~(2). Note that $\min(\mathcal{S})$ cannot be larger than~$\{v\}$ in this case by Proposition~\ref{prop:2-trans_implies_not_uniscalar}
\eqref{prop:2-trans_implies_not_uniscalar1a}. If $\mathcal{S}$ contains hyperbolic elements, then the axis of each one passes through $v$. Put 
$$
I = \left\{ (w,v) \colon \hbox{ there is }h\in \mathcal{S} \hbox{ that translates along }(w,v)\right\}.
$$
Then $I\ne \emptyset$ and, since any hyperbolic $h\in \mathcal{S}$ translates out along some edge~$f$ and~$\bar{f}$ then cannot belong to $I$ by Lemma~\ref{lem:scale(prod(hyperbolics nontrivial intersection incoherent))}, we also have that $I$ is a proper subset of $E(v)$. Hence we are in case~(3). Note that, in this case, $2$-transitivity of the action of $G$ on $\partial T$ ensures, by Proposition~\ref{prop:2-trans_implies_not_uniscalar}
\eqref{prop:2-trans_implies_not_uniscalar1}, that for every $e_1\in I$ and $e_2\in {I}^*$ there is a hyperbolic $h\in \mathcal{S}$ that translates along $e_1$ and $e_2$. Hence $\min(\mathcal{S}) = \{v\}$ unless $I$ or ${I}^*$ consists of a single edge $e$, in which case $\min(\mathcal{S}) = \{e,\bar{e}\}$.

To show 
that 
each of the sets listed 
is indeed 
a maximal \scalar\ semigroup it suffices to show that none is contained in any of the others. \\
\underline{Case~(1):}  The semigroup~$G_p$ fixing the midpoint,~$p$, of an edge cannot be contained in any semigroup~$G_v$, $\GIv$,~$\Gom$ or~$\omG$ because~$G_p$ contains an inversion and the other semigroups do not. Hence~$G_p$ is maximal. (This case only occurs if~$G$ contains an inversion.)\\
\underline{Case~(2):} The semigroup~$G_v$ cannot be contained in~$G_p$,~$\Gom$ or~$\omG$  because, if it were,~$G_v$ would fix a midpoint of an edge or an end in addition to~$v$, which it does not. Similarly, if~$G_v$ were contained in $\GIw$ for some~$w$ and~$I$, then ~$w$ would have to equal~$v$ but~$G_v$ is transitive on~$E(v)$ while~$\GIv$ is not. Hence~$G_v$ is not contained in any other semigroup and is maximal. \\
\underline{Case~(3):} The semigroup~$\GIv$ contains hyperbolic automorphisms of~$T$ and so is not contained in $G_p$ or $G_v$ for any midpoint~$p$ or vertex~$v$. It is not contained in~$\Gom$ or $\omG$ for any~$\omega\in\partial T$ because it does not fix~$\omega$. Hence~$\GIv$ is maximal. \\
\underline{Case~(4):} The semigroups~$\Gom$ and~$\omG$ cannot be contained in any of the others because $\min(\Gom)$ and $\min(\omG)$ are empty while the minimizing sets of the other semigroups are not. Hence~$\Gom$ and~$\omG$ are maximal.
\end{proof}

%
%

\medskip 

\begin{remark}
~\\[-3ex]
\begin{enumerate}
\item
The maximal \scalar\ semigroups 
of type~(1)--(3) 
in Theorem~\ref{thm:class(scalar-semigroups)} 
are open
because 
they contain 
an open subgroup of $G$,  
while 
the maximal \scalar\ semigroups 
of type~(4) 
are not. 
\item
The semigroups 
of type~(3), 
despite being open, do not contain 
subgroups tidy
for their hyperbolic elements unless either $|I|=1$ or $|I^*|=1$. 
\item
The inverse map on maximal \scalar\ semigroups, see~Proposition~\ref{prop:involution}, satisfies $\Gom^{-1} = \omG$ for every~$\omega\in\partial T$ and $\GIv^{-1} = \GJv$, where~$J = \bar{I}^*$, for every $v\in V(T)$ and $I\subset E(v)$.
\end{enumerate}
\end{remark}

\begin{remark}\label{lem:axes-branching}\label{cond:axes-branching} 
The condition that~$G$ act $2$-transitively on $\partial T$ is required at two places in the characterisation of the maximal \scalar\ semigroups in~$G$. The first is to derive the formula~\eqref{eq:scale_calc} for the scale found in Proposition~\ref{prop:2-trans_implies_not_uniscalar}, and the second is to guarantee that distinct sets $I,J\subset E(v)$ determine distinct semigroups~$\GIv$ and~$\GJv$. Formula~\eqref{eq:scale_calc} relates the scale of~$h$ to translation distance but a formally weaker and more complicated condition suffices for this. It is enough to know, for each vertex~$v$, that for every hyperbolic~$h$ with $v\in \min(h)$, the valency of~$v$ in the subtree~$G_{\epsilon_+(h)}$ does not depend on~$h$. This holds, for example, if $G_{\epsilon_+(h)}$   
is the whole tree for every hyperbolic~$h$. However, it is possible that these weaker conditions, together with the hypothesis that~$G$ does not fix an end, already imply that~$G$ acts $2$-transitively on $\partial T$.
\end{remark}

\section{Conclusion}
\label{sec:conclusion}
This last section points out relationships between the maximal \scalar\ semigroups listed in Theorem~\ref{thm:class(scalar-semigroups)} and how they might be relevant for the possible construction of a `structure space' for a general totally disconnected, locally compact group~$G$ that encodes important information about~$G$. 
 
When the maximal \scalar\ semigroup is~$G_p$, the fixator of an edge, or~$G_v$, the fixator of a vertex, it is in fact a maximal compact, open subgroup of~$G$. These subgroups are thus identified with the defining features of the tree, which can be recovered from them because the incidence relations can be detected from the indices $[G_p : G_p\cap G_v]$ and $[G_v : G_p\cap G_v]$. 

The other maximal \scalar\ semigroups in~$G$ are not subgroups and  not compact but are identified with features of the tree, namely, vertices and ends. Although the index of an intersection does not make sense for these semigroups, relative directions or orientations of features in the tree can be detected from whether the intersection contains hyperbolic elements. For example, $\Gom\cap \Gomp$ does not contain any hyperbolics if $\omega\ne \sigma$ while, on the other hand, $\Gom\cap \ompG$ always contains hyperbolics. Similarly, whether $\GIv\cap \Gom$ contains hyperbolic elements depends on whether the ray $[v,\omega)$ passes through an edge in~$I^*$. Since 
$$
\{ U_{(v,I)} \colon v\in V(T),\ \emptyset \ne I\subsetneq E(v) \}\subset \mathcal{P}(\partial T),
$$
where
$$
U_{(v,I)}  = \{ \omega\in \partial T \colon \GIv\cap \Gom\hbox{ contains a hyperbolic element}\}, 
$$
is a base of neighborhoods for the usual topology on $\partial T$, these semigroups and the relations between them detect more than the sets of vertices or ends of~$T$. 

As stated in the introduction, our purpose in investigating \scalar\ semigroups is to see whether they may be used to construct, for each general totally disconnected, locally compact group, a natural relational or geometric object on which it acts. Indeed, Theorem~\ref{thm:class(scalar-semigroups)} shows that both the vertices and edges of~$T$ as well as its ends can be detected from maximal \scalar\ semigroups, thus suggesting in general that more than one object that might be constructed. Furthermore, as discussed in the previous couple of paragraphs, relations between the semigroups correspond to incidence relations between vertices and edges, or a topology on the set of ends, and it is hoped that a similar correspondence will continue to hold more generally.

Theorem~\ref{thm:class(scalar-semigroups)} also indicates some limitation on the information about~$G$ that can be expected to be recovered from the maximal \scalar\ semigroups in~$G$. Its hypotheses are satisfied by the group $\operatorname{PSL}(2,\mathbb{Q}_p)$ acting on its Bruhat-Tits tree, see~\cite[chapter~II]{trees}, in which every vertex has valency~$p+1$, and also by~$\hbox{Aut}^+(T_{p+1})$, the group of automorphisms of the homogeneous $(p+1)$-valent tree,~$T_{p+1}$, that leave the subsets in the standard bipartition of~$V(T)$ invariant. These groups are quite different. For example, $\operatorname{PSL}(2,\mathbb{Q}_p)$ acts strictly $3$-transitively on $\partial T_{p+1}$, whereas $\hbox{Aut}^+(T_{p+1})$ acts $3$-transitively but not strictly so. More abstractly, the contraction subgroups of $\operatorname{PSL}(2,\mathbb{Q}_p)$ are closed while in $\hbox{Aut}^+(T_{p+1})$ they are not, see~\cite{contraction(aut(tdlcG))}. Yet the map $$
\mathcal{S} \mapsto \mathcal{S}\cap \operatorname{PSL}(2,\mathbb{Q}_p),\ \  (\mathcal{S}\hbox{ maximal \scalar\ semigroup in } \hbox{Aut}^+(T_{p+1}))
$$ 
is a bijection to the maximal \scalar\ semigroups in $\operatorname{PSL}(2,\mathbb{Q}_p)$. Therefore maximal \scalar\ semigroups and the relations between them do not distinguish between $\operatorname{PSL}(2,\mathbb{Q}_p)$ and $\hbox{Aut}^+(T_{p+1})$.

The groups $\operatorname{PSL}(2,\mathbb{Q}_p)$ and $\hbox{Aut}^+(T_{p+1})$ are distinguished however by their local structure as studied in~\cite{CapReWi}, where a structure space on which a totally disconnected, locally compact group acts is produced from subgroups normal in a compact, open subgroup of the group. Compact, open subgroups of $\operatorname{PSL}(2,\mathbb{Q}_p)$ are just-infinite, that is, they have no proper, non-trivial normal subgroups, while compact, open subgroups of $\hbox{Aut}^+(T_{p+1})$ have an infinite lattice of closed, normal subgroups.  Information about a group that might be captured by a description of its maximal \scalar\ semigroups in contrast appears to be large scale and complementary to this local structure. 

The example provided by the pair of groups $\operatorname{PSL}(2,\mathbb{Q}_p)$ and $\hbox{Aut}^+(T_{p+1})$ illustrates another feature that would be desirable in any `structure space' that might be constructed from the maximal \scalar\ subgroups of a group $G$. Just as the Bruhat-Tits tree, that is,~$T_{p+1}$, is constructed directly from $\operatorname{PSL}(2,\mathbb{Q}_p)$ and has automorphism group that is a totally disconnected, locally compact group, a structure space constructed from maximal \scalar\ semigroups in~$G$ could be expected to itself have an automorphism group, call it~$\widetilde{G}$, that is totally disconnected and locally compact. Just as $\operatorname{PSL}(2,\mathbb{Q}_p)$ embeds as a closed subgroup of $\hbox{Aut}(T_{p+1})$, there would then be a natural continuous homomorphism $G\to \widetilde{G}$. This homomorphism would not in general be injective because, for example, when $G$ is discrete there is only one maximal \scalar\ semigroup, namely,~$G$ itself. However, under certain hypotheses such as, for example, that $G$ be non-discrete and simple, the homomorphism could be expected to be injective.

\end{document}